\documentclass[10pt]{amsart}
\usepackage{fancyhdr}
\usepackage{stmaryrd}
\usepackage{calrsfs}

\usepackage{amsmath}
\usepackage[nospace,noadjust]{cite}
\usepackage{amsfonts}
\usepackage{amssymb,enumerate}
\usepackage{amsthm}
\usepackage{cite}
\usepackage{comment}
\usepackage{color}
\usepackage[all]{xy}
\usepackage{hyperref}
\usepackage{lineno}
\usepackage{stmaryrd}

\usepackage{mathtools}
\usepackage{bm}
\usepackage{graphicx}
\usepackage{psfrag}
\usepackage{url}

\usepackage{fancyhdr}
\usepackage{tikz-cd}

\usepackage{dirtytalk}
\usepackage{float}
\usepackage{mathrsfs}

\newtheorem*{maintheorem*}{Main Theorem}
\newtheorem{theorem}{Theorem}[section]
\newtheorem{prop}[theorem]{Proposition}
\newtheorem{question}[theorem]{Question}

\newtheorem{lemma}[theorem]{Lemma}

\theoremstyle{definition}
\newtheorem{definition}[theorem]{Definition}
\newtheorem{remark}[theorem]{Remark}
\newtheorem{example}[theorem]{Example}
\numberwithin{equation}{section}

\newcommand{\cc}{\mathbb{C}}

\newcommand{\nn}{\mathbb{N}}
\newcommand{\qq}{\mathbb{Q}}
\newcommand{\rr}{\mathbb{R}}
\newcommand{\zz}{\mathbb{Z}}

\newcommand{\uu}{\mathcal{U}}
\newcommand{\mcd}{\text{mcd}}
\newcommand{\ord}{\text{ord}}
\newcommand{\red}{\text{red}}

\providecommand\ldb{\llbracket}
\providecommand\rdb{\rrbracket}

\hypersetup{
	pdftitle={FD},
	colorlinks=true,
	linkcolor=blue,
	citecolor=cyan,
	urlcolor=wine
}



\begin{document}
	
	\mbox{}
	\title{On the arithmetic of polynomial semidomains}
	
	\author{Felix Gotti}
	\address{Department of Mathematics\\MIT\\Cambridge, MA 02139}
	\email{fgotti@mit.edu}
	
	\author{Harold Polo}
	\address{Department of Mathematics\\University of Florida\\Gainesville, FL 32611}
	\email{haroldpolo@ufl.edu}
	
\date{\today}

\begin{abstract}
	 A subset $S$ of an integral domain $R$ is called a semidomain provided that the pairs $(S,+)$ and $(S, \cdot)$ are semigroups with identities. The study of factorizations in integral domains was initiated by Anderson, Anderson, and Zafrullah in 1990, and this area has been systematically investigated since then. In this paper, we study the divisibility and arithmetic of factorizations in the more general context of semidomains. We are specially concerned with the ascent of the most standard divisibility and factorization properties from a semidomain to its semidomain of (Laurent) polynomials. As in the case of integral domains, here we prove that the properties of satisfying the ascending chain condition on principal ideals, having bounded factorizations, and having finite factorizations ascend in the class of semidomains. We also consider the ascent of the property of being atomic and that of having unique factorization (none of them ascends in general). Throughout the paper we provide several examples aiming to shed some light upon the arithmetic of factorizations of semidomains.
\end{abstract}
\medskip

\maketitle

%\tableofcontents

\bigskip
%%%%%%%%%%%%
%%%%%%%%%%%%
\section{Introduction}
\label{sec:intro}

A subset of an integral domain containing $0$ and $1$ and closed under both addition and multiplication is called a semidomain. As for integral domains, we say that a semidomain is atomic if every nonzero element that is not a multiplicative unit factors into irreducibles. The first systematic study of factorizations in the context of integral domains was carried out by Anderson, Anderson, and Zafrullah in~\cite{AAZ90}, where they not only introduced and studied the bounded and the finite factorization properties but also investigated further factorization properties, including being atomic, satisfying the ascending chain condition on principal ideals (ACCP), and being factorial or half-factorial. With the same properties in mind, here we study the arithmetic of the more general class of semidomains, putting special emphasis on whether such properties ascend from a semidomain to its semidomain of (Laurent) polynomials.
\smallskip

A commutative semiring $S$ is a nonempty set endowed with two compatible binary operations denoted by `$+$' and `$\cdot$' such that $(S,+)$ and $(S, \cdot)$ are commutative semigroups with identities. Commutative semirings consisting of nonnegative real numbers (under the standard addition and multiplication) are called positive semirings. Clearly, every positive semiring is a semidomain. The atomicity of positive semirings consisting of rational numbers was first studied by Chapman, Gotti, and Gotti~\cite{CGG20} and then by Albizu-Campos, Bringas, and Polo~\cite{ABP21}. Positive semirings were also studied by Baeth and Gotti~\cite{BG20} in connection with factorizations of certain square matrices. Several examples of positive semirings were recently given by Baeth, Chapman, and Gotti in~\cite{BCG21}, where for the first time additive and multiplicative factorizations in positive semirings were considered simultaneously.
\smallskip

The algebraic structures of central interest in this paper are semidomains of polynomials and semidomains of Laurent polynomials. The arithmetic of certain semidomains of (Laurent) polynomials has been considered in the literature in the past few years. In~\cite{CCMS09}, Cesarz et al. studied the elasticity of the semidomain $\rr_{\ge 0}[x]$, where $\rr_{\ge 0}$ is the nonnegative ray of $\rr$. In addition, methods to factorize polynomials in the semidomain $\nn_0[x]$ were investigated by Brunotte~\cite{hB13}. More recently, Campanini and Facchini~\cite{CF19} provided a more systematic investigation of the arithmetic of factorizations in~$\nn_0[x]$. More generally, semigroup semirings were considered by Ponomarenko in~\cite{vP15} from the factorization point of view. Finally, it is worth emphasizing that the study of the algebraic structure of commutative semirings has been significantly motivated by the applications and connections of commutative semirings to theoretical computer science (see, for instance, \cite{DKV09} and the references therein).
\smallskip

The positive semirings we obtain as homomorphic images of semidomains of polynomials and semidomains of Laurent polynomials have also been investigated recently. Atomic and factorization aspects of the additive structure of the semidomain $\nn_0[\alpha]$, where $\alpha$ is a positive algebraic number, were first investigated by Correa-Morris and Gotti~\cite{CG22}. More recently, the elasticity and the omega-primality of $\nn_0[\alpha]$ have been considered by Jiang, Li, and Zhu~\cite{JLZ22}. Also, the atomicity of the multiplicative structure of $\nn_0[\alpha]$ was studied in~\cite{CCMS09}, where the authors proved that $\nn_0[\alpha]$ has full infinite elasticity for reasonable quadratic algebraic integers~$\alpha$. On the other hand, Zhu~\cite{sZ22} has recently studied several factorization aspects of the homomorphic image $\nn_0[\alpha^{\pm 1}]$ of the semidomain of Laurent polynomials $\nn_0[x^{\pm 1}]$, where~$\alpha$ is a positive algebraic number.
\smallskip

In Section~\ref{sec:background}, we briefly revise the definitions and terminology relevant to this paper. 
\smallskip

In Section~\ref{sec:atomicity}, we study the property of being atomic in the context of semidomains. It was proved by M. Roitman~\cite[Proposition~1.1]{mR93} that atomicity ascends from any integral domain $R$ to the ring of polynomials $R[x]$ provided that the set of coefficients of any indecomposable polynomial over $R$ has a maximal common divisor. We start Section~\ref{sec:atomicity} generalizing this result: we prove that under the same condition, atomicity ascends from any semidomain to its semidomain of (Laurent) polynomials. Given the relevant role played by the existence of maximal common divisors in the ascent of atomicity, we also prove that the existence of maximal common divisors is a condition that ascends from any semidomain to its (Laurent) polynomial semidomain. We also justify why we do not consider the ascent of the studied factorization properties to power series semidomains: we show in Example~\ref{ex:power series extensions} that the semidomain of power series $\nn_0\ldb x \rdb$ is not atomic (even though $\nn_0$ satisfies the unique factorization property).
\smallskip

In Section~\ref{sec:ACCP}, we investigate ACCP. First, we prove that the property of satisfying ACCP ascends from any semidomain to its semidomain of (Laurent) polynomials. In~\cite[Theorem~1.3]{aG74}, Grams constructed a celebrated example of an atomic domain that does not satisfy ACCP, disproving a wrong assertion made by Cohn~\cite{pC68} on the equivalence of the condition of being atomic and that of satisfying ACCP. Further examples have been given by Zaks~\cite{aZ82} and Roitman~\cite{mR93} and, more recently, by Boynton and Coykendall~\cite{BC19} and also by Gotti and Li~\cite{GL21,GL23,GL22}. In Example~\ref{ex:atomic positive semiring without ACCP}, we construct a positive semiring (which, clearly, cannot be an integral domain) that is atomic but does not satisfy ACCP.
\smallskip

In Section~\ref{sec:bounded and finite factorization properties}, we consider both the bounded and the finite factorization properties. For the rest of this section, let $S$ be a semidomain. We say that~$S$ is a bounded factorization semidomain (BFS) if there is a function $\ell \colon S \setminus \{0\} \to \nn_0$ that is only zero on units and satisfies that $\ell(rs) \ge \ell(r) + \ell(s)$ for all $r,s \in S \setminus \{0\}$. On the other hand, we say that~$S$ is a finite factorization semidomain (FFS) if every nonzero element has finitely many divisors up to associates. Both notions are extensions of the corresponding notions introduced in~\cite{AAZ90} in the setting of integral domains. It is well known that the bounded and the finite factorization properties ascend from an integral domain to its ring of (Laurent) polynomials (see \cite[Propositions~2.5 and~5.3]{AAZ90} and \cite[Corollary~2.2]{AAZ92}). In Theorems~\ref{thm:BF transfers} and~\ref{thm:FF transfers}, we extend these results by proving that the same properties ascend from any semidomain to its semidomain of (Laurent) polynomials.
\smallskip

In Section~\ref{sec:factoriality}, we study factorial and length-factorial semidomains. As for integral domains, we say that~$S$ is a factorial semidomain or a unique factorization semidomain (UFS) provided that every nonzero element of~$S$ has a unique factorization into irreducibles. On the other hand, following the terminology introduced by Chapman et al. in~\cite{CCGS21}, we say that $S$ is a length-factorial semidomain (LFS) if $S$ is atomic and any two distinct factorizations of the same element of~$S$ have distinct numbers of irreducible factors (counting repetitions). Length-factoriality was first studied by Coykendall and Smith in~\cite{CS11}, where they used the same notion to characterize UFDs. Length-factoriality has been more recently investigated by several authors in~\cite{CCGS21, GZ21,fG20} in the more general context of commutative monoids. We prove that integral domains are the only semidomains where the unique factorization and the length-factorial properties ascend to their corresponding (Laurent) polynomial semidomains: this is obtained as a combination of Proposition~\ref{prop:LFS polynomial semiring has coefficients in an integral domain} and Theorem~\ref{UFD}. We conclude the paper with a few words on half-factoriality and the elasticity of semidomains of polynomials.

\bigskip
%%%%%%%%%%%%
%%%%%%%%%%%%
\section{Preliminaries}
\label{sec:background}

In this section, we introduce the notation and terminology we shall be using later. For a comprehensive background on factorization theory and semiring theory, the reader can consult \cite{GH06} and \cite{JG1999}, respectively. Following standard notation, we let $\zz, \qq$, and $\rr$ denote the sets of integers, rational numbers, and real numbers, respectively. Additionally, we let $\nn$ denote the set of positive integers, while we let $\nn_0$ denote the set of nonnegative integers. Given $r \in \rr$ and $S \subseteq \rr$, we set $S_{<r} \coloneqq \{s \in S \mid s < r\}$; we define $S_{>r}$ and $S_{\geq r}$ in a similar way. For $m,n \in \nn_0$, we denote by $\llbracket m,n \rrbracket$ the discrete interval from $m$ to $n$; that is, $\llbracket m,n \rrbracket \coloneqq \{k \in \zz \mid m \leq k \leq n\}$. Finally, for $q \in \qq_{>0}$,  the relatively prime positive integers $n$ and $d$ satisfying that $q = \frac nd$ are denoted here by $\mathsf{n}(q)$ and $\mathsf{d}(q)$, respectively.

\medskip
%%%%%%%%%%%%%%%%%%%%%%
\subsection{Monoids and Factorizations}

Throughout this paper, a \emph{monoid}\footnote{The standard definition of a monoid does not assume the cancellative and the commutative conditions.} is defined to be a semigroup with identity that is cancellative and commutative. As we are primarily interested in the multiplicative structure of certain semirings, unless otherwise specified we will use multiplicative notation for monoids. Let $M$ be a monoid with identity~$1$. A subsemigroup of $M$ is called a \emph{submonoid} if it contains $1$. We set $M^{\bullet} \coloneqq M \setminus \{1\}$, and we let $\mathcal{U}(M)$ denote the group of units (i.e., invertible elements) of~$M$. In addition, we let $M_{\red}$ denote the quotient $M/\mathcal{U}(M)$, which is also a monoid. The monoid $M$ is \emph{reduced} provided that $\mathcal{U}(M)$ is the trivial group, in which case we naturally identify $M_{\red}$ with $M$. The \emph{Grothendieck group} of~$M$ is an abelian group $\mathcal{G}(M)$ for which there exists a monoid homomorphism $\iota \colon M \to \mathcal{G}(M)$ satisfying the following universal property: for any monoid homomorphism $f \colon M \to G$, where $G$ is an abelian group, there exists a unique group homomorphism $g \colon \mathcal{G}(M) \to G$ such that $f = g \circ \iota$. The Grothendieck group of a monoid is unique up to isomorphism. For a subset $S$ of $M$, we let $\langle S \rangle$ denote the smallest submonoid of $M$ containing~$S$, and~$S$ is a \emph{generating set} of~$M$ provided that $M = \langle S \rangle$.

For $b,c \in M$, we say that $b$ \emph{divides} $c$ \emph{in} $M$ if there exists $a \in M$ such that $c = ab$, in which case we write $b \mid_M c$, dropping the subscript precisely when $M = (\nn, \times)$. Two elements $b,c \in M$ are \emph{associates} if $b \mid_M c$ and $c \mid_M b$. A submonoid $N$ of $M$ is \emph{divisor-closed} if for each $b \in N$ and $d \in M$ the relation $d \mid_M b$ implies that $d \in N$. Let $S$ be a nonempty subset of $M$. An element $d \in M$ is a \emph{common divisor} of~$S$ provided that $d \mid_M s$ for all $s \in S$. A common divisor $d$ of $S$ is a \emph{greatest common divisor} of~$S$ if $d$ is divisible by all common divisors of $S$. Also, a common divisor $d$ of $S$ is a \emph{maximal common divisor} of $S$ if every common divisor of the set $\{s/d \mid s \in S\}$ belongs to $\uu(M)$. We let $\gcd_M(S)$ (resp., $\mcd_M(S)$) denote the set consisting of all greatest common divisors (resp., maximal common divisors) of $S$, dropping the subindex $M$ in the introduced notation when we see no danger of ambiguity. The monoid $M$ is a \emph{GCD-monoid} (resp., an \emph{MCD-monoid}) provided that every finite nonempty subset of~$M$ has a greatest common divisor (resp., maximal common divisor).

An element $a \in M \setminus \uu(M)$ is an \emph{atom} if for all $b,c \in M$ the equality $a = bc$ implies that either $b \in \uu(M)$ or $c \in \uu(M)$. We let $\mathcal{A}(M)$ denote the set of all atoms of $M$. The monoid $M$ is \emph{atomic} if every element of $M \setminus \uu(M)$ can be written as a (finite) product of atoms. One can readily check that~$M$ is atomic if and only if $M_{\red}$ is atomic. A subset $I$ of $M$ is an \emph{ideal} of $M$ provided that $I M \subseteq I$ or, equivalently, $I M = I$. An ideal $I$ of $M$ is \emph{principal} if $I = b M$ for some $b \in M$. The monoid~$M$ satisfies the \emph{ascending chain condition on principal ideals} (\emph{ACCP}) if every increasing sequence of principal ideals of $M$ (under inclusion) becomes constant from one point on. It is well known and not hard to verify that monoids satisfying ACCP are atomic.

Assume now that $M$ is atomic. We let $\mathsf{Z}(M)$ denote the free (commutative) monoid on $\mathcal{A}(M_{\red})$. The elements of $\mathsf{Z}(M)$ are called \emph{factorizations}, and if $z = a_1 \cdots a_\ell \in \mathsf{Z}(M)$ for $a_1, \ldots, a_\ell \in \mathcal{A}(M_{\red})$, then $\ell$ is called the \emph{length} of $z$ and is denoted by $|z|$. Let $\pi \colon \mathsf{Z}(M) \to M_{\red}$ be the unique monoid homomorphism satisfying that $\pi(a) = a$ for all $a \in\mathcal{A}(M_{\red})$. For each $b \in M$, the following sets associated to $b$ are fundamental in the study of factorization theory:
\begin{equation} \label{eq:sets of factorizations/lengths}
	\mathsf{Z}_M(b) \coloneqq \pi^{-1}(b \mathcal{U}(M)) \subseteq \mathsf{Z}(M) \hspace{0.6 cm}\text{ and } \hspace{0.6 cm}\mathsf{L}_M(b) \coloneqq \{|z| : z \in\mathsf{Z}_M(b)\} \subseteq \nn_0.
\end{equation}
We drop the subscript $M$ in~\eqref{eq:sets of factorizations/lengths} whenever the same monoid is clear from the context. Following the terminology in~\cite{AAZ90} and~\cite{fHK92}, we say that $M$ is a \emph{finite factorization monoid} (\emph{FFM}) if $\mathsf{Z}(b)$ is finite for all $b \in M$, and we say that~$M$ is a \emph{bounded factorization monoid} (\emph{BFM}) if $\mathsf{L}(b)$ is finite for all $b \in M$. It follows directly from the definitions that every FFM is a BFM and, by virtue of \cite[Corollary~1.3.3]{GH06}, every BFM satisfies ACCP. Following the terminology in~\cite{aZ76}, we say that~$M$ is a \emph{half-factorial monoid} (\emph{HFM}) if $|\mathsf{L}(b)| = 1$ for all $b \in M$. Finally, $M$ is a \emph{unique factorization monoid} (\emph{UFM}) if $|\mathsf{Z}(b)| = 1$ for all $b \in M$. It follows from the definitions that every UFM is an HFM and also that every HFM is a BFM. The atomic classes defined in this paragraph are those represented in the square of the diagram in Figure~\ref{fig:AAZ's atomic chain for monoids}. The same diagram was introduced by Anderson, Anderson, and Zafrullah in~\cite{AAZ90} and, since then, it has been used as a methodology to study atomicity and the phenomenon of multiple factorizations. Finally, we follow the terminology in~\cite{CCGS21} and say that~$M$ is a \emph{length-factorial monoid} (\emph{LFM}) if for all $b \in M$ and $z, z' \in \mathsf{Z}(b)$, the equality $|z| = |z'|$ implies $z = z'$. It is clear that every UFM is an LFM.
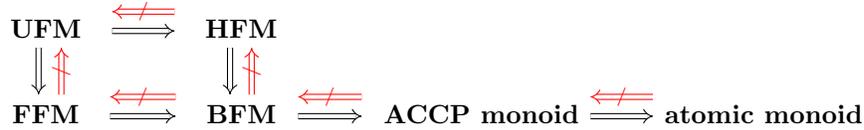
\begin{figure}[h]
	\begin{tikzcd}%[cramped]
		\textbf{ UFM } \ \arrow[r, Rightarrow] \arrow[red, r, Leftarrow, "/"{anchor=center,sloped}, shift left=1.7ex] \arrow[d, Rightarrow, shift right=1ex] \arrow[red, d, Leftarrow, "/"{anchor=center,sloped}, shift left=1ex]& \ \textbf{ HFM } \arrow[d, Rightarrow, shift right=0.6ex] \arrow[red, d, Leftarrow, "/"{anchor=center,sloped}, shift left=1.3ex] \\
		\textbf{ FFM } \ \arrow[r, Rightarrow]	\arrow[red, r, Leftarrow, "/"{anchor=center,sloped}, shift left=1.7ex] & \ \textbf{ BFM } \arrow[r, Rightarrow]  \arrow[red, r, Leftarrow, "/"{anchor=center,sloped}, shift left=1.7ex] & \textbf{ ACCP monoid}  \arrow[r, Rightarrow] \arrow[red, r, Leftarrow, "/"{anchor=center,sloped}, shift left=1.7ex]  & \textbf{atomic monoid}
	\end{tikzcd}
	\caption{The implications in the diagram determine inclusions among the subclasses of atomic monoids that we have previously mentioned. The diagram also emphasizes (with red marked arrows) that each of the mentioned inclusions is proper.}
	\label{fig:AAZ's atomic chain for monoids}
\end{figure}

\medskip
%%%%%%%%%%%%%
\subsection{Semirings}

A \emph{commutative semiring} $S$ is a nonempty set endowed with two binary operations denoted by `$+$' and `$\cdot$' and called \emph{addition} and \emph{multiplication}, respectively, such that the following conditions hold:
\begin{itemize}
	\item $(S,+)$ is a monoid with its identity element denoted by $0$;
	\smallskip
	
	\item $(S, \cdot)$ is a commutative semigroup with an identity element denoted by $1$ with $1 \neq 0$;
	\smallskip
	
	\item $b \cdot (c+d)= b \cdot c + b \cdot d$ for all $b, c, d \in S$.
\end{itemize}

Let $S$ be a commutative semiring. Since the operation of addition in $S$ is cancellative, the distributive law ensures that $0 \cdot b = 0$ for all $b \in S$. Throughout this paper, for any $b,c \in S$, we write $b c$ instead of $b \cdot c$ when there seems to be no risk of ambiguity. In the typical definition of a `semiring' $S$, one does not assume that the semigroup $(S,+)$ is cancellative. However, we do so here because our semirings of interest have cancellative additive structures. On the other hand, it is worth observing that a more general notion of a `semiring' $S$ does not assume that the semigroup $(S, \cdot)$ is commutative. However, once again, this more general type of algebraic objects are not of interest in the scope of this paper. Accordingly, from now on we will use the single term \emph{semiring} to refer to a commutative semiring, tacitly assuming the commutativity of both operations. A subset $S'$ of~$S$ is a \emph{subsemiring} provided that $(S',+)$ is a submonoid of $(S,+)$ that is closed under multiplication and contains~$1$. Observe that every subsemiring of $S$ is a semiring. The set $S[x]$ (resp., $S[x^{\pm 1}]$) consisting of all polynomials  (resp., Laurent polynomials) with coefficients in $S$ is also a semiring under the standard addition and multiplication of polynomials (resp., Laurent polynomials). We call $S[x]$ (resp., $S[x^{\pm 1}]$) the \emph{semiring of polynomials} (resp., \emph{semiring of Laurent polynomials}) over~$S$.

\begin{definition}
	We say that a semiring $S$ is a \emph{semidomain} provided that $S$ is a subsemiring of an integral domain.
\end{definition}
 
Let $S$ be a semidomain. Then $(S \setminus \{0\}, \cdot)$ is a monoid, which we denote by $S^*$ and call the \emph{multiplicative monoid} of $S$. In Example~\ref{ex:semiring not integral}, we show a semiring that is not a semidomain but still its subset of nonzero elements is a multiplicative monoid. In order to reuse notation from ring theory, we refer to the units of the monoid $(S,+)$ as \emph{invertible elements} of the semidomain $S$, so that we can refer to the units of the multiplicative monoid~$S^*$ simply as \emph{units} of $S$, avoiding any risk of ambiguity. Also, following standard notation from ring theory, we let $S^\times$ denote the group of units of $S$, letting $\uu(S)$ refer to the additive group of invertible elements of $S$. In addition, we write $\mathcal{A}(S)$ instead of $\mathcal{A}(S^*)$ for the set of atoms of the multiplicative monoid~$S^*$ (in this paper, we do not consider the set of atoms of the additive monoid of a semidomain, except briefly in Example~\ref{ex:power series extensions}). Finally, for any $b,c \in S$ such that $b$ divides~$c$ in~$S^*$, we write $b \mid_S c$ instead of $b \mid_{S^*} c$, and the term (greatest, maximal) common divisor in the context of semidomains are to be understood in the multiplicative monoid $S^*$.
\smallskip

For the next example, we need the following lemma.

\begin{lemma} \label{lem:characterization of integral semirings}
	For a semiring $S$, the following conditions are equivalent.
	\begin{enumerate}
		\item[(a)] $S$ is a semidomain.
		\smallskip
		
		\item[(b)] The multiplication of $S$ extends to the Grothendieck group $\mathcal{G}(S)$ of $(S,+)$ turning $\mathcal{G}(S)$ into an integral domain.
	\end{enumerate}
\end{lemma}

\begin{proof}
	(b) $\Rightarrow$ (a): This is clear.
	\smallskip
	
	(a) $\Rightarrow$ (b): Let $S$ be a semidomain, and suppose that $S$ is embedded into an integral domain~$R$. We can identify the Grothendieck group $\mathcal{G}(S)$ of $(S,+)$ with the subgroup $\{r - s \mid r,s \in S\}$ of the underlying additive group of $R$. It is easy to see then that $\mathcal{G}(S)$ is closed under the multiplication it inherits from $R$, and it contains the multiplicative identity because $0,1 \in S$. Hence $\mathcal{G}(S)$ is an integral domain having $S$ as a subsemiring. 
\end{proof}

\begin{example} \label{ex:semiring not integral}
	Notice that the set $S := \{(0,0)\} \cup (\nn \times \nn)$ is a monoid with the usual component-wise addition, and it is closed under the usual component-wise multiplication with multiplication identity $(1,1)$. Hence $S$ is a semiring. Observe, on the other hand, that any extension of the multiplication of~$S$ to $\mathcal{G}(S) = \zz \times \zz$ making the latter a commutative ring must respect the identity $(1,0) (0,1) = (0,0)$ and, therefore, will not turn $\mathcal{G}(S)$ into an integral domain. Hence it follows from Lemma~\ref{lem:characterization of integral semirings} that $S$ is not a semidomain.
\end{example}

We say that $S$ is \emph{atomic} (resp., satisfies \emph{ACCP}) if its multiplicative monoid $S^*$ is atomic (resp., satisfies ACCP), while we say that $S$ is a \emph{BFS} (resp., \emph{FFS}, \emph{HFS}, \emph{LFS}, \emph{UFS}) provided that~$S^*$ is a BFM (resp., FFM, HFM, LFM, UFM). In addition, we call $S$ a \emph{GCD-semidomain} (resp., an \emph{MCD-semidomain}) if $S^*$ is a GCD-monoid (resp., an MCD-monoid). Note that when $S$ is an integral domain, we recover the usual definition of a UFD (as well as the definitions of a BFD, an FFD, and an HFD, which are now standard notions in atomicity and factorization theory). Although a semidomain $S$ can be embedded into an integral domain~$R$, the former may not inherit any atomic property from the latter as, after all, the integral domain~$\qq[x]$ is a UFD but it contains as a subring the integral domain $\zz + x\qq[x]$, which is not even atomic.

Let $R$ be an integral domain containing $S$ as a subsemiring. Then the semiring of polynomials $S[x]$ is a subsemiring of $R[x]$, and so $S[x]$ is also a semidomain. The elements of $S[x]$ are, in particular, polynomials in $R[x]$ and, as a result, all the standard terminology for polynomials can be applied to elements of $S[x]$, including \emph{degree}, \emph{order}, \emph{leading coefficient}, etc. Observe that $S^*$ is a divisor-closed submonoid of $S[x]^*$ and, therefore, $S[x]^\times = S^\times$ and $\mathcal{A}(S[x]) \cap S = \mathcal{A}(S)$. Following the terminology in~\cite{mR93}, we say that a nonzero polynomial in $S[x]$ is \emph{indecomposable} if it cannot be written as a product of two non-constant polynomials in $S[x]$. Similarly, the semiring of Laurent polynomials $S[x^{\pm 1}]$ over~$S$ is also a semidomain. Observe that the multiplicative set $\{sx^n \mid s \in S^* \text{ and } n \in \zz\}$ is actually a divisor-closed submonoid of $S[x^{\pm 1}]^*$ and, as a consequence,
\[
	S[x^{\pm 1}]^\times = \{sx^n \mid s \in S^\times \text{ and } n \in \zz \}.
\]

Following the terminology in~\cite{BCG21}, we say that a subsemiring of $\rr$ (under the standard addition and multiplication) is a \emph{positive semiring} if it consists of nonnegative numbers. The fact that underlying additive monoids of positive semirings are reduced makes them more tractable. The reader can check the recent paper~\cite{BCG21} for several examples of positive semirings illustrating several aspects of their atomic structure. The class of semidomains clearly contains those of integral domains and positive semirings. As the following example illustrates, integral domains and positive semirings account for all semidomains that can be embedded into~$\qq$.

\begin{example} \label{ex:semidomains of rationals}
	Suppose that $S$ is a semidomain that is a subsemiring of $\qq$, and assume that $S$ is not a positive semiring. Since $S$ is not a positive semiring, it must contain a negative rational. By virtue of \cite[Theorem~2.9]{rG84} any additive submonoid of $\qq$ containing both a negative rational and a positive rational must be a subgroup of $\qq$. As a result, the additive monoid of $S$ is a subgroup of~$\qq$. Thus, $S$ is a subring of~$\qq$, and so an integral domain.
\end{example}

In general, there are semidomains that are neither integral domains nor positive semirings.

\begin{example} \label{ex: integral semiring that is not an information algebra}
	Consider the semidomain $\nn_0[\pm \alpha, \beta]$, where $\alpha, \beta \in \rr_{>0}$ are algebraically independent elements over~$\qq$. Observe that the monoid $(\nn_0[\pm \alpha, \beta],+)$ is not reduced. On the other hand, since~$\alpha$ and~$\beta$ are algebraically independent elements over $\qq$, we see that $-\beta \not\in \nn_0[\pm \alpha, \beta]$ and, therefore, $(\nn_0[\pm \alpha, \beta],+)$ is not a group. Hence the semidomain $\nn_0[\pm \alpha, \beta]$ is neither an integral domain nor a positive semiring.
\end{example}

\bigskip
%%%%%%%%%%
%%%%%%%%%%
\section{Atomicity}
\label{sec:atomicity}

Back in 1990, Anderson, Anderson, and Zafrullah posed the question of whether the property of being atomic ascends from any integral domain to its polynomial ring \cite[Question~1]{AAZ90}. Three years later, Roitman provided a negative answer for that question, constructing in~\cite{mR93} examples of atomic domains whose polynomial rings are not atomic (in a parallel direction, examples of atomic monoids with non-atomic monoid domains were constructed by Coykendall and Gotti in~\cite{CG19}). Roitman also found a sufficient condition for the ascent of atomicity; this is \cite[Proposition~1.1]{mR93}. We proceed to generalize this last result to the context of semidomains.

\begin{theorem} \label{thm:atomicity transfers}
	For a semidomain $S$, the following statements are equivalent.
	\begin{enumerate}
		\item[(a)] $S$ is atomic and, for any indecomposable polynomial $\sum_{i=0}^n c_i x^i \in S[x]^*$, the set $\emph{\mcd}(c_0, \dots, c_n)$ is nonempty.
		\smallskip
		
		\item[(b)] $S[x]$ is atomic.
		\smallskip
		
		\item[(c)] $S[x^{\pm 1}]$ is atomic. 
	\end{enumerate}
\end{theorem}

\begin{proof}
	(a) $\Rightarrow$ (b): Assume, towards a contradiction, that $S[x]$ is not atomic. Let $f$ be a minimum-degree nonunit polynomial in $S[x]^*$ that does not factor into irreducibles. Then $f$ must be indecomposable and, as $S$ is atomic, $\deg f \ge 1$. By assumption, we can write $f = c g$, where $c$ is a maximal common divisor of the set of coefficients of~$f$. As any common divisor of the set of coefficients of $g$ belongs to $S^\times$, the fact that $g$ is indecomposable implies that $g$ is irreducible in~$S[x]$. This, along with the fact that $f$ does not factor into irreducibles in $S[x]$, guarantees that $c$ is a nonzero nonunit of $S$ that does not factor into irreducibles. However, this contradicts that $S$ is atomic.
	\smallskip
	
	(b) $\Rightarrow$ (c): We claim that $\mathcal{A}(S[x]) \subseteq \mathcal{A}(S[x^{\pm 1}]) \cup S[x^{\pm 1}]^\times$. To argue this, take $f \in \mathcal{A}(S[x])$, and assume that $f \notin S[x^{\pm 1}]^\times = \{ux^n \mid u \in S^\times \text{ and } n \in \zz\}$. Now write $f = gh$ for some $g,h \in S[x^{\pm 1}]$. After replacing $g$ and $h$ by some of their associates in $S[x^{\pm 1}]$, we can assume that $g,h \in S[x]$. Since $f \in \mathcal{A}(S[x])$, either $g$ or $h$ belongs to $S[x]^\times$ and, therefore, $S[x]^\times = S^\times \subseteq S[x^{\pm 1}]^\times$ implies that $f \in \mathcal{A}(S[x^{\pm 1}])$. Finally, it is clear that the fact that $S[x]$ is atomic, in tandem with the established inclusion $\mathcal{A}(S[x]) \subseteq \mathcal{A}(S[x^{\pm 1}]) \cup S[x^{\pm 1}]^\times$, ensures that $S[x^{\pm 1}]$ is atomic.
	\smallskip
	
	(c) $\Rightarrow$ (a): The multiplicative submonoid $M := \{sx^n \mid s \in S^* \text{ and } n \in \zz\}$ of $S[x^{\pm 1}]^*$ is atomic because it is a divisor-closed submonoid. Since $S^*_{\text{red}} \cong M_{\text{red}}$, we conclude that $S$ is atomic.
	
	Now suppose, for the sake of a contradiction, that there exists an indecomposable nonzero polynomial $f = \sum_{i=0}^n c_i x^i \in S[x]$ such that $\mcd(c_0, \ldots, c_n) = \emptyset$. As $f$ is indecomposable and $\mcd(c_0, \ldots, c_n)$ is empty, $\text{ord}\, f = 0$. In addition, the fact that $\mcd(c_0, \dots, c_n)$ is empty ensures that $\deg f \ge 1$. Hence~$f$ is a nonzero nonunit in $S[x^{\pm 1}]$, and so we can write $f = a_1 \cdots a_\ell$ for some $a_1, \dots, a_\ell \in \mathcal{A}(S[x^{\pm 1}])$. After replacing $a_1, \dots, a_\ell$ for some of their associates in $S[x^{\pm 1}]$, we can assume that they all belong to $S[x]$. Because $f$ is indecomposable, we can further assume that $\deg a_i = 0$ for every $i \in \ldb 1,\ell-1 \rdb$ and $\deg a_\ell = \deg f$. Since $a_\ell \in \mathcal{A}(S[x^{\pm 1}]) \cap S[x]$ and $\ord \, a_\ell = 0$, we see that $a_\ell \in \mathcal{A}(S[x])$, which implies that every common divisor of the set of coefficients of $a_\ell$ belongs to $S^\times$. Hence $a_1 \cdots a_{\ell-1}$ must belong to $\mcd(c_0, \dots, c_n)$, which is a contradiction.
\end{proof}

As a consequence of Theorem~\ref{thm:atomicity transfers}, we obtain that if a semidomain $S$ is an atomic MCD-semidomain, then both extensions $S[x]$ and $S[x^{\pm 1}]$ are atomic. Next we show that in this case $S[x]$ and $S[x^{\pm 1}]$ are also MCD-semidomains. 

\begin{prop} \label{prop: MCD transfers}
	For a semidomain $S$, the following statements are equivalent.
	\begin{enumerate}
		\item[(a)] $S$ is an MCD-semidomain.
		\smallskip
		
		\item[(b)] $S[x]$ is an MCD-semidomain.
		\smallskip
		
		\item[(c)] $S[x^{\pm 1}]$ is an MCD-semidomain.
	\end{enumerate}
\end{prop}

\begin{proof}
	(a) $\Rightarrow$ (b): Suppose that $S$ is an MCD-semidomain, and consider the set
	\[
		T = \big\{ (f_1, \ldots, f_n) \in S[x]^n \mid n \in \nn_{\ge 2} \quad \text{and} \quad \mcd(f_1, \ldots, f_n) = \emptyset \big\}.
	\]
	Assume, towards a contradiction, that $T$ is nonempty, and take $(g_1, \ldots, g_m) \in T$ such that $\sum_{i = 1}^m \deg \, g_i$ is as small as it can possibly be. Observe that $(g_1/g, \dots, g_m/g) \in T$ for any common divisor $g$ of $\{g_1, \dots, g_m\}$ in $S[x]$. For $f = \sum_{i=0}^n c_ix^i \in S[x]$, set $C_f \coloneqq \{c_0, \ldots, c_n\}$. Set $C := \bigcup_{i = 1}^m C_{g_i}$. As~$S$ is an MCD-semidomain, $\mcd(C)$ is nonempty. Take $d \in \mcd(C)$. Note that $d \mid_{S[x]} g_i$ for every $i \in \ldb 1,m \rdb$. Now suppose that $u \in S[x]$ is a common divisor of the set $\{ g_1/d, \ldots, g_m/d \}$. It follows from the minimality of $\sum_{i = 1}^m \deg \, g_i$ that $u \in S$. Therefore~$u$ is a common divisor of the set $C/d$ in $S$, and the fact that $d$ belongs to $\text{mcd}(C)$ guarantees that $u \in S^\times$. Hence $d$ is a maximal common divisor of $\{g_1, \dots, g_m\}$ in $S[x]$, contradicting that $(g_1, \dots, g_m) \in T$.
	\smallskip
	
	(b) $\Rightarrow$ (c): Assume that $S[x]$ is an MCD-semidomain, and fix $f_1, \dots, f_k \in S[x^{\pm 1}]$.  Because $S[x]$ is an MCD-semidomain, we can pick $d \in \mcd_{S[x]}(g_1, \dots, g_k)$, where $g_i := x^{-\text{ord} \, f_i} f_i$ for every $i \in \ldb 1,k \rdb$. Note that $d$ is a common divisor of $\{f_1, \dots, f_k\}$ in $S[x^{\pm 1}]$. Let $f$ be a common divisor of $\{f_1/d, \dots, f_k/d\}$ in $S[x^{\pm 1}]$. Because the polynomial $g := x^{-\text{ord} f} f \in S[x]$ divides $g_i/d$ in $S[x]$ for every $i \in \ldb 1,k \rdb$, the fact that $d \in \mcd_{S[x]}(g_1, \dots, g_k)$ implies that $g \in S[x]^\times = S^\times$. Hence $f \in S[x^{\pm}]^\times$, and so $d$ is a maximal common divisor of $\{f_1, \dots, f_k\}$ in $S[x^{\pm 1}]$. Thus, $S[x^{\pm 1}]$ is an MCD-semidomain.
	\smallskip
	
	(c) $\Rightarrow$ (a): Suppose that $S[x^{\pm 1}]$ is an MCD-semidomain. Since $M = \{sx^n \mid s \in S^* \text{ and } n \in \zz\}$ is a divisor-closed submonoid of $S[x^{\pm 1}]^*$ satisfying that $\mathcal{U}(M) = S[x^{\pm 1}]^{\times}$, we see that $M$ is an MCD-monoid. This immediately implies that $S$ is an MCD-semidomain because $M$ and $S^*$ have isomorphic reduced monoids.
\end{proof}

Let $S$ be a semidomain, and consider the Grothendieck group $\mathcal{G}(S)$ of $(S,+)$ as an integral domain having $S$ as a subsemiring (see Lemma~\ref{lem:characterization of integral semirings}). The subset of the ring of formal power series $\mathcal{G}(S)\ldb x \rdb$ consisting of those power series with coefficients in $S$ is a semidomain, which we call the \emph{semidomain of formal power series over} $S$ and denote by $S\ldb x \rdb$. In \cite{mR2000}, Roitman proved that atomicity does not ascend from an integral domain to its ring of formal power series. Next we show that the semidomain of formal power series over $S$ is not necessarily atomic, even when $S$ is taken to be the UFS $\nn_0$.

\begin{example}  \label{ex:power series extensions}
	Let us argue that $\nn_0\ldb x \rdb$ is not atomic. It is clear that $\nn_0\ldb x \rdb$ is reduced. Suppose, towards a contradiction, that $\nn_0\ldb x \rdb$ is atomic. Set $f := \sum_{i=0}^{\infty} x^i$. Observe that we can write $f = gh$, where either $g \coloneqq \sum_{i = 0}^{\infty} c_i x^i$ or $h \coloneqq \sum_{i = 0}^{\infty} d_i x^i$ is an irreducible of $\nn_0\ldb x \rdb$ that is not a polynomial. It is not hard to verify that neither $g$ nor $h$ is equal to $(1 - x^k)^{-1}$ for any $k \in \nn$. Clearly, the inequality $c_i + d_i \leq 1$ holds for every $i \in \nn$. We can assume, without loss of generality, that $c_0 = \cdots = c_t = c_{(k + 1)(t + 1)}= 1$ and $d_0 = d_{t + 1} = \cdots = d_{k(t + 1)} = 1$ for some $t, k \in \nn$ (note that, implicitly, we are also assuming that $c_{j} = 0$ for every $j \in \llbracket t + 1,k(t + 1) + t \rrbracket$ and $d_s = 0$ for every $s \in \{n \in \nn \mid n < k(t + 1) \text{ and } t + 1 \nmid n\}$). Set
	\[
		C = \Big\{ n \in \nn : t + 1 \nmid n \text{ and } d_n = 1 \Big\} \bigcup \Big\{ n \in \nn : t + 1 \mid n,\, c_n = 1, \text{ and } c_{n + j} = 0 \text{ for some } j \in \llbracket 1,t \rrbracket \Big\}.
	\]
	
	\noindent \textit{Claim 1:} $C$ is empty.
	\smallskip
	
	\noindent \textit{Proof of Claim 1:} Suppose, by way of contradiction, that $C$ is nonempty. Let $m$ be the minimal element of $C$. If $t + 1 \nmid m$, then $m = k'(t + 1) + j$ for some $j \in \llbracket 1,t \rrbracket$ and $k' \in \nn_{\geq k}$. It is easy to see that the term $x^{k'(t + 1)}$ does not show up in $h$ and, by the minimality of $m$, this term does not show up in $g$ either. As a result, there exist $r,\ell \in (t + 1)\nn$ such that $r + \ell = k'(t + 1)$, $c_\ell = 1$, $d_r = 1$, and $\max(r,\ell) < k'(t + 1)$. The minimality of $m$ now implies that $c_{\ell + j} = 1$, which contradicts that $d_m = 1$. Hence $m = k'(t + 1)$ for some $k' \in \nn_{\geq k + 1}$, which implies the existence of $j \in \llbracket 1,t \rrbracket$ with $c_{m + j} = 0$. Without loss of generality, we can assume that $c_{m + j - i} = 1$ for every $i \in \llbracket 1,j \rrbracket$. It is easy to verify that $d_{m + j} = 0$. Consequently, there exist $r,\ell \in \nn_{< m}$ with $r + \ell = m + j$, $c_\ell = 1$, and $d_r = 1$. The minimality of $m$ ensures that $t + 1 \mid r$ which, in turn, implies that $\ell = k''(t + 1) + j$ for some $k'' \in \nn$. Since $c_m = 1$, both equalities $c_{k''(t + 1)} = 0$ and $d_{k''(t + 1)} = 0$ hold. Again, there exist $r', \ell' \in \nn_{< k''(t + 1)}$ with $r' + \ell' = k''(t + 1)$, $c_{\ell'} = 1$, and $d_{r'} = 1$. The minimality of $m$ guarantees that $t + 1 \mid r'$, which implies that $c_{\ell' + j} = 1$, but this contradicts that $c_\ell = 1$. Thus, Claim~1 follows.
	\smallskip
	
	Since~$C$ is empty, if $c_{n(t + 1)} = 1$ for some $n \in \nn$, then $c_{n(t + 1) + j} = 1$ for all $j \in \llbracket 1,t \rrbracket$. We proceed to argue the following claim.
	\smallskip
	
	\noindent \textit{Claim 2:} If $c_{n(t + 1)} = 0$ for some $n \in \nn$, then $c_{n(t + 1) + j} = 0$ for all $j \in \llbracket 0,t \rrbracket$.
	\smallskip
	
	\noindent \textit{Proof of Claim 2:} Suppose towards a contradiction that $c_{n(t + 1)} = 0$ for some $n \in \nn$ and $c_{n(t + 1) + j} = 1$ for some $j \in \llbracket 1,t \rrbracket$. It is not hard to see that $d_{n(t + 1)} = 0$, which implies that there exist $r,\ell \in \nn_{<n(t + 1)}$ such that $r + \ell = n(t + 1)$, $c_\ell = 1$, and $d_r = 1$. Since $C$ is empty, we see that $t + 1 \mid r$ (and, consequently, $t + 1 \mid \ell$). Then $c_{\ell + j} = 1$. However, this contradicts the fact that $c_{n(t + 1) + j} = 1$, and so Claim~2 is established.
	\smallskip
	
	Therefore it follows from Claim~2 that $1 + x +\dots + x^t$ divides $g$ in $\nn_0\ldb x \rdb$. Assume now that, for some $m \in \nn_0$, we can write 
	\[
		h = \sum_{i = 0}^{m} x^{r_i(t + 1)} \sum_{i=0}^k x^{i(t+1)} + \sum_{i = n}^{\infty} d_{i(t + 1)}x^{i(t + 1)},
	\]
	where $r_0 = 0$, $r_{i + 1} > r_i + k$ for $i \in \llbracket 0, m - 1 \rrbracket$, and $n > r_m + k$. Now assume, without loss of generality, that $d_{n(t + 1)} = 1$. 
	\smallskip
	
	\noindent \textit{Claim 3:} $d_{(n + j)(t + 1)} = 1$ for every $j \in \llbracket 0,k \rrbracket$. 
	\smallskip
	
	\noindent \textit{Proof of Claim 3:} Suppose, towards a contradiction, that there exists $j \in \llbracket 1,k \rrbracket$ such that $d_{(n + j)(t + 1)} = 0$, and assume that $j$ is minimal. Since $c_{(k + 1)(t + 1)} = d_{k(t + 1)} = 1$ and $n > k$, we have $c_{(n + j)(t + 1)} = 0$. Consequently, there exist $r,\ell \in \nn$ such that $r + \ell = (n + j)(t + 1)$, $c_r = 1$, and $d_\ell = 1$. Clearly, we can write $\ell = (r_u + v)(t + 1)$, where $u \in \llbracket 0,m \rrbracket$ and $v \in \llbracket 0,k \rrbracket$. Observe that $v < j$; otherwise $d_{(r_u + v - j)(t + 1)} = 1$, which implies that $d_{n(t + 1)} = 0$. Thus,
	\begin{equation} \label{eq: different factors same monomial}
		x^{n(t + 1)} \cdot x^{(k + 1)(t + 1)} = x^{(r_u + v + k - j + 1)(t + 1)}\cdot x^{r},
	\end{equation}
	where $1 \leq v + k - j + 1 \leq k$. Since $n > r_m + k \ge r_u + (v + k - j + 1)$, the equality~\eqref{eq: different factors same monomial} generates a contradiction. Hence Claim~3 follows.
	\smallskip

	By induction, we obtain that $1 + x^{t + 1} + \cdots + x^{k(t + 1)}$ divides $h$ in $\nn_0\ldb x \rdb$. We have, therefore, argued that both $g$ and $h$ are divisible by some non-constant polynomials in $\nn_0\ldb x \rdb$, which contradicts the assumption made on $g$ and $h$ at the beginning of the proof. Hence $\nn_0\ldb x \rdb$ is not atomic.
\end{example}

\begin{remark}
	As $\nn_0\ldb x \rdb$ is not atomic even though $\nn_0$ is a UFS, most factorization properties do not ascend, in general, from a semidomain to its semidomain of formal power series. In particular, the statement of Theorem~\ref{thm:ACCP transfers} is not longer true after replacing either $S[x]$ or $S[x^{\pm 1}]$ by $S\ldb x \rdb$.
\end{remark}

We conclude this section with a related question that the authors were unable to answer while working on this paper. Following the terminology in~\cite{nL19}, we say that a semidomain $S$ is \emph{nearly atomic} provided that there exists $s \in S$ such that $st$ factors into atoms for each nonunit $t \in S^*$.

\begin{question}
	Is the semidomain $\nn_0\ldb x \rdb$ nearly atomic?
\end{question}

\bigskip
%%%%%%%%%%%%%%%%%%%%%%%%%
%%%%%%%%%%%%%%%%%%%%%%%%%
\section{Ascending Chains of Principal Ideals}
\label{sec:ACCP}

Unlike the property of being atomic, it is well known that the property of satisfying ACCP ascends from each integral domain to its ring of (Laurent) polynomials (this is not the case in the more general context of commutative rings with identity, as shown by Heinzer and Lantz in~\cite{HL94}). In the next theorem, we generalize this result to the context of semidomains. If $S$ is a semidomain and $f \in S[x]^*$, then we let $c(f)$ denote the leading coefficient of $f$.

\begin{theorem} \label{thm:ACCP transfers}
	For a semidomain $S$, the following statements are equivalent.
	\begin{enumerate}
		\item[(a)] $S$ satisfies ACCP.
		\smallskip
		
		\item[(b)] $S[x]$ satisfies ACCP.
		\smallskip
		
		\item[(c)] $S[x^{\pm 1}]$ satisfies ACCP.
	\end{enumerate}
\end{theorem}

\begin{proof}
	(a) $\Rightarrow$ (b): Let $(f_n S[x])_{n \in \nn}$ be an ascending chain of principal ideals in the semidomain $S[x]$. Since $\deg f_n \ge \deg f_{n+1}$ for every $n \in \nn$, we can choose $N_1 \in \nn$ with $\deg f_n = \deg f_{N_1}$ for every $n \ge N_1$. On the other hand, for each $n \in \nn$, the fact that $f_{n+1}$ divides $f_n$ in $S[x]$ implies that $c(f_{n + 1})$ divides $c({f_{n}})$ in $S$. Therefore $(c(f_n) S)_{n \in \nn}$ is an ascending chain of principal ideals in $S$. Since $S$ satisfies ACCP, there exists $N_2 \in \nn$ such that $c(f_n)$ and $c(f_{N_2})$ are associates in $S^*$ for every $n \geq N_2$. After setting $N := \max\{N_1, N_2\}$, we can take a sequence $(s_n)_{n \in \nn}$ with terms in $S^*$ and a sequence $(u_n)_{n \in \nn}$ with terms in $S^\times$ such that $f_n = s_n f_N$ and $c(f_n) = u_n c(f_N)$ for every $n \ge N$. Thus, for each $n \in \nn$ with $n \ge N$,
	\[	
		u_n c(f_N) = c(f_n) = c(s_n f_N) = s_n c(f_N),
	\]
	which implies that $s_n \in S^\times$.  As a result, $f_n S[x] = f_N S[x]$ for every $n \ge N$, and so the ascending chain of principal ideals $(f_n S[x])_{n \in \nn}$ stabilizes. Hence $S[x]$ satisfies ACCP.
	\smallskip
	
	(b) $\Rightarrow$ (c): Now assume that $S[x]$ satisfies ACCP. Let $(x^{k_n}f_n S[x^{\pm 1}])_{n \in \nn}$ be an ascending chain of principal ideals of $S[x^{\pm 1}]$, where $(k_n)_{n \in \nn}$ is a sequence of integers and $f_n \in S[x]$ satisfies $\text{ord} \, f_n = 0$ for every $n \in \nn$. Take a sequence $(\ell_n)_{n \in \nn}$ of integers and a sequence $(g_n)_{n \in \nn}$ with terms in $S[x]$ satisfying $\text{ord} \, g_n = 0$ and $x^{k_n}f_n = (x^{k_{n+1}} f_{n+1})(x^{\ell_{n+1}} g_{n+1})$ for every $n \in \nn$. As $f_n = f_{n+1} g_{n+1}$ for each $n \in \nn$, we see that $(f_n S[x])_{n \in \nn}$ is an ascending chain of principal ideals in $S[x]$. Since $S[x]$ satisfies ACCP, $(f_n S[x])_{n \in \nn}$ must stabilize, and so there is an $N \in \nn$ such that $f_n$ and $f_{n+1}$ are associates in $S[x]$ for every $n \ge N$. This implies that $g_{n+1} = f_n/f_{n+1} \in S[x]^\times \subseteq S[x^{\pm 1}]^\times$ for every $n \ge N$. Thus, $x^{k_n}f_n$ and $x^{k_{n+1}} f_{n+1}$ are associates in $S[x^{\pm 1}]$ for every $n \ge N$, and so $(x^{k_n}f_n S[x^{\pm 1}])_{n \in \nn}$ stabilizes. Hence $S[x^{\pm 1}]$ satisfies ACCP.
	\smallskip
	
	(c) $\Rightarrow$ (a): Suppose that $S[x^{\pm 1}]$ satisfies ACCP. Since $M = \{sx^n \mid s \in S^* \text{ and } n \in \zz\}$ is a divisor-closed submonoid of $S[x^{\pm 1}]^*$, the former also satisfies ACCP. This immediately implies that $S$ satisfies ACCP because $M$ and $S^*$ have isomorphic reduced monoids.
\end{proof}

It is well known that every monoid satisfying ACCP is atomic, and so the same statement holds for semidomains. As we have mentioned in the introduction, the converse does not hold even in the class of integral domains. We will construct in Example~\ref{ex:atomic positive semiring without ACCP} an atomic positive semiring that does not satisfy ACCP. First, we need the following lemma.

\begin{lemma} \label{lem:molecules of S_r}
	For $r \in \qq \cap (0,1)$ with $\mathsf{n}(r) \ge 2$, consider the additive monoid $S_r \coloneqq \langle r^n \mid n \in \nn_0 \rangle$.  Take $x \in S_r^*$, and write $x = \sum^n_{i=0} c_ir^i$ for coefficients $c_0, \ldots, c_n \in \mathbb{N}_0$. If $c_i < \mathsf{n}(r)$ for every $i \in \llbracket 0,n \rrbracket$, then $|\mathsf{Z}(x)| = 1$.
\end{lemma}

\begin{proof}
	Set $z \coloneqq \sum_{i = 0}^n c_ir^i \in \mathsf{Z}(x)$. Suppose towards a contradiction that there exists a factorization $z' \in \mathsf{Z}(x)$ such that $z' \neq z$. It follows from \cite[Lemma 3.1(1)]{CGG20} that $z$ is the factorization of minimum length of~$x$. Hence there exists a sequence of factorizations $z_1, \ldots,z_t \in \mathsf{Z}(x)$ with $z_1 = z'$ and $z_t = z$ such that for each $j \in \ldb 1, t-1 \rdb$, the factorization $z_{j+1}$ can be obtained from $z_j$ by applying the identity $\mathsf{d}(r)r^{k+1} = \mathsf{n}(r)r^k$ (see \cite[Remark~3.4]{hP23}). However, this would imply that, for some $i \in \llbracket 0,n \rrbracket$, the inequality $c_i \geq \mathsf{n}(r)$ holds, which is a contradiction. Hence $|\mathsf{Z}(x)| = 1$.
\end{proof}

\begin{example} \label{ex:atomic positive semiring without ACCP}
	Take $r \in \qq \cap (0,1)$ with $\mathsf{n}(r) \ge 2$, and consider the additive monoid $S_r \coloneqq \langle r^n \mid n \in \nn_0 \rangle$. By \cite[Corollary~4.4]{CGG20}, the monoid $S_r$ is atomic and does not satisfy ACCP. We proceed to argue that $S_r$ is an MCD-monoid. Take $s_1, \ldots, s_k \in S_r$ for some $k \in \nn$. For each $i \in \llbracket 1,k \rrbracket$, we can write $s_i = \sum_{j = 0}^{n_i} c_{i,j} r^j$, where $c_{i,j} \in \nn_0$ for all $i \in \llbracket 1,k \rrbracket$ and $j \in \llbracket 0, n_i \rrbracket$. By virtue of the identity $\mathsf{n}(r)r^n = \mathsf{d}(r)r^{n + 1}$, there is no loss of generality in assuming that $c_{i,j} < \mathsf{n}(r)$ for all $i \in \llbracket 1,k \rrbracket$ and $j \in \llbracket 0, n_i - 1 \rrbracket$. In addition, the same identity allows us to assume that $n_1 = \cdots = n_k$.  Set $d \coloneqq (\min_{1 \leq i \leq k} c_{i,{n_1}})r^{n_1}$. Clearly, $d$ is a common divisor of $s_1, \ldots, s_k$ in $S_r$. Observe that, for some $i \in \llbracket 1,k \rrbracket$, the equality $s_i - d = \sum_{j = 0}^{n_1 - 1} c_{i,j} r^j$ holds, where $c_{i,j} < \mathsf{n}(r)$ for each $j \in \llbracket 0, n_1 - 1 \rrbracket$. It follows now from Lemma~\ref{lem:molecules of S_r} that the element $s_i - d$ has finitely many nonzero divisors in~$S_r$, which implies that $\mcd(s_1, \ldots, s_k)$ is nonempty. Hence $S_r$ is an MCD-monoid.
	
	Let us now consider the additive monoid $E(S_r) \coloneqq \langle e^s \mid s \in S_r \rangle$, where $e$ is the Euler number. It follows from Lindemann-Weierstrass Theorem that $E(S_r)$ is the free monoid on the set $M := \{e^s \mid s \in S_r\}$. Note that $E(S_r)$ is closed under multiplication and, consequently, it is a positive semiring (cf. \cite[Example~4.15]{BG20}). Observe that $\min E(S_r)^* = 1$ and, therefore, $E(S_r)^\times = \{1\}$. Since the multiplicative submonoid $M$ of $E(S_r)^*$ is isomorphic to the additive monoid $S_r$, the monoid $M$ does not satisfy ACCP which, in turn, implies that $E(S_r)$ does not satisfy ACCP (the property of satisfying ACCP transfers from a reduced monoid to all its submonoids). To argue that the positive semiring $E(S_r)$ is atomic, take a nonzero nonunit element $f := c_1 e^{s_1} + \cdots + c_k e^{s_k} \in E(S_r)$, where $c_1, \dots, c_k \in \nn$ and $s_1, \ldots, s_k \in S_r$. After taking $s \in \mcd(s_1, \ldots, s_k)$ and letting $d$ be the greatest common divisor of $c_1, \ldots, c_k$ in $\nn$, we can write
	\[
		f = d e^s \bigg( \frac{c_1}d e^{s_1 - s} + \cdots + \frac{c_k}d e^{s_k - s} \bigg),
	\]
	where $0$ is the only common divisor of $s_1 - s, \ldots, s_k - s$ in $S_r$. As a consequence of Lindemann-Weierstrass Theorem, both $\nn$ and $M$ are divisor-closed submonoids of the multiplicative monoid $E(S_r)^*$ and, therefore, both $\mathbb{P}$ and $\{e^a \mid a \in \mathcal{A}(S_r)\}$ are contained in $\mathcal{A}(E(S_r))$. Thus, the fact that $\nn$ and $S_r$ are atomic immediately implies that $de^s$ factors into irreducibles in $E(S_r)$. Set $g \coloneqq  \frac{c_1}d e^{s_1 - s} + \cdots + \frac{c_k}d e^{s_k - s}$, and write $g = f_1 \cdots f_m$ for some nonunit elements $f_1, \dots, f_m \in E(S_r)$. As $0$ is the only common divisor of $s_1 - s, \ldots, s_k - s$ in $S_r$, no element of the form $e^t$ with $t \in S_r$ divides~$g$ in $E(S_r)$ and, therefore, $f_i \ge 2$ for any $i \in \ldb 1,m \rdb$. Then $m \leq \log_2(f_1 \cdots f_m) = \log_2 \big(  \frac{c_1}d e^{s_1 - s} + \cdots + \frac{c_k}d e^{s_k - s} \big)$. Hence, after assuming that $m$ is as large as it can possibly be, we obtain that $f_1 \cdots f_m$ is a factorization of $g$ in $E(S_r)$, and so~$f$ factors into irreducibles. Thus, $E(S_r)$ is atomic.
\end{example}

We take another look at the semidomain $\nn_0\ldb x \rdb$, now from the ACCP perspective.

\begin{example}
	We have already argued in Example~\ref{ex:power series extensions} that the semidomain $\nn_0\ldb x \rdb$ is not atomic. Thus, it cannot satisfy ACCP. Let us identify an ascending chain of principal ideals that does not stabilize. For each $k \in \nn_0$, set $f_k \coloneqq \sum_{n = 0}^{\infty} x^{n\cdot 2^k}$. Observe that $f_k = (1 + x^{2^k})f_{k+1}$ for every $k \in \nn_0$, which implies that $(f_k \nn_0\ldb x \rdb)_{k \in \nn_0}$ is an ascending chain of principal ideals of $\nn_0\ldb x \rdb$. Since $\nn_0 \ldb x \rdb$ is reduced, the inclusion $f_k \nn_0\ldb x \rdb \subseteq f_{k+1} \nn_0 \ldb x \rdb$ is proper for every $k \in \nn_0$. Consequently, the ascending chain of principal ideals $(f_k \nn_0\ldb x \rdb)_{k \in \nn_0}$ does not stabilize.
\end{example}

The main results we have established in this section are illustrated in the diagram of Figure~\ref{fig:atomic/ACCP transfers}.

\begin{figure}[h]
	\begin{tikzcd}[cramped]
		S \textbf{ \, satisfies ACCP } \  \arrow[r, Rightarrow]	\arrow[red, r, Leftarrow, "/"{anchor=center,sloped}, shift left=1.7ex] \arrow[d, Leftrightarrow] & \ S \textbf{ \, is atomic } \arrow[d, Leftarrow, shift right=1ex] \arrow[red, d, Rightarrow, "/"{anchor=center,sloped}, shift left=1ex]  \\ %& S \textbf{ satisfies ACCP } \arrow[d, Leftrightarrow] \\
		S[x] \textbf{ satisfies ACCP } \ \arrow[d, Leftrightarrow] \arrow[r, Rightarrow]	\arrow[red, r, Leftarrow, "/"{anchor=center,sloped}, shift left=1.7ex] & \ S[x] \textbf{ is atomic } \arrow[d, Leftrightarrow] \\%\arrow[r, Rightarrow]  \arrow[red, r, Leftarrow, "/"{anchor=center,sloped}, shift left=1.7ex] %& S[x] \textbf{ satisfies ACCP}  %\arrow[r, Rightarrow] \arrow[red, r, Leftarrow, "/"{anchor=center,sloped}, shift left=1.7ex]  &
		S[x^{\pm 1}] \textbf{ satisfies ACCP } \ \arrow[r, Rightarrow] \arrow[red, r, Leftarrow, "/"{anchor=center,sloped}, shift left=1.7ex] & \ S[x^{\pm 1}] \textbf{ is atomic } %\arrow[d, Leftrightarrow]
	\end{tikzcd}
	\caption{As proved in Theorems~\ref{thm:atomicity transfers} and~\ref{thm:ACCP transfers} the vertical implications in the above diagram hold for every semidomain $S$. Grams' construction in~\cite[Section~1]{aG74} and Roitman's examples in~\cite[Section~5]{mR93} confirm that neither the horizontal implications nor the top-right implication are reversible, which is illustrated by red marked arrows.}
	\label{fig:atomic/ACCP transfers}
\end{figure}
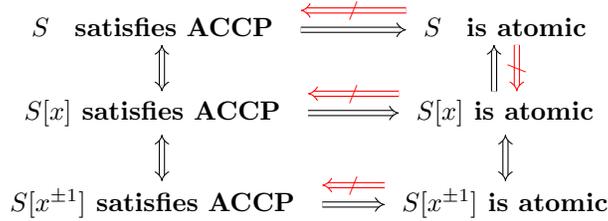

\bigskip
%%%%%%%%%%%%%%%%%%%%%%%%%%%%%%%
%%%%%%%%%%%%%%%%%%%%%%%%%%%%%%%
\section{The Bounded and Finite Factorization Properties}
\label{sec:bounded and finite factorization properties}

In this section, we study the bounded and the finite factorization properties. Specifically, we prove that both properties ascend from a semidomain $S$ to the semidomains $S[x]$ and $S[x^{\pm 1}]$. Let us start with a useful characterization of BFMs.

\begin{definition}
	Given a monoid $M$, a function $\ell\colon M \to \nn_0$ is a \emph{length function} of $M$ if it satisfies the following two properties:
	\begin{enumerate}
		\item[(i)] $\ell(u) = 0$ if and only if $u \in \mathcal{U}(M)$;
		\smallskip
		
		\item[(ii)] $\ell(bc) \geq \ell(b) + \ell(c)$ for every $b,c \in M$.
	\end{enumerate}
\end{definition} 

The following result is well known.

\begin{prop}\cite[Theorem~1]{fHK92}
	A monoid $M$ is a BFM if and only if there is a length function $\ell\colon M \to \nn_0$.
\end{prop}

We are now in a position to discuss the results of this section.

\begin{theorem} \label{thm:BF transfers}
	For a semidomain $S$, the following statements are equivalent.
	\begin{enumerate}
		\item[(a)] $S$ is a BFS.
		\smallskip
		
		\item[(b)] $S[x]$ is a BFS.
		\smallskip
		
		\item[(c)] $S[x^{\pm 1}]$ is a BFS.
	\end{enumerate}
\end{theorem}

\begin{proof}
	(a) $\Rightarrow$ (b): Assume that $S$ is a BFS. Then there exists a length function $\ell \colon S^* \to \nn_0$. Let us argue now that the function $\ell_x \colon S[x]^* \to \nn_0$ given by $\ell_x(f) = \ell(c(f)) + \deg f$ is also a length function. Since $f \in S[x]^*$ is a unit if and only if $f = c(f)$ and $c(f) \in S^\times$, we see that for each $f \in S[x]^*$ the equality $\ell_x(f) = 0$ holds if and only if $f \in S[x]^\times$. Using now the fact that $\ell$ is a length function of $S^*$, for any $f, g \in S[x]^*$ we see that
	\[
		\ell_x(fg) = \ell(c(fg)) + \deg fg \ge (\ell(c(f)) + \deg f) + (\ell(c(g)) + \deg g) = \ell_x(f) + \ell_x(g).
	\]
	Therefore the map $\ell_x$ is a length function of $S[x]^*$, which implies that $S[x]$ is a BFS. 
	\smallskip
	
	(b) $\Rightarrow$ (c): Suppose now that $S[x]$ is a BFS, and let $\ell \colon S[x]^* \to \nn_0$ be a length function of $S[x]^*$. Proving that $S[x^{\pm 1}]$ is a BFS amounts to showing that the map
	\[
		\bar{\ell} \colon S[x^{\pm 1}]^* \to \nn_0 \quad \text{defined by} \quad \bar{\ell}(f) = \ell \Big( \frac{f}{x^{\text{ord} \, f}} \Big)
	\]
	 is a length function. For each $f \in S[x^{\pm 1}]^*$, we observe that $\bar{\ell}(f) = 0$ if and only if $f/x^{\text{ord} \, f}$ is a unit of $S[x]$, which happens precisely when $f$ is a unit in $S[x^{\pm 1}]$. In addition, for all $f, g \in S[x^{\pm 1}]^*$,
	\[
		\bar{\ell}(fg) = \ell \Big( \frac{fg}{x^{\text{ord} \, fg}} \Big) = \ell \Big( \frac{f}{x^{\text{ord} \, f}} \cdot  \frac{g}{x^{\text{ord} \, g}}  \Big) \ge \ell \Big( \frac{f}{x^{\text{ord} \, f}} \Big) + \ell \Big( \frac{g}{x^{\text{ord} \, g}}  \Big) = \bar{\ell}(f) + \bar{\ell}(g).
	\]
	As a consequence, $\bar{\ell}$ is a length function, and so $S[x^{\pm 1}]$ is a BFS.
	\smallskip
	
	(c) $\Rightarrow$ (a): This follows from the fact that $\{sx^n \mid s \in S^* \text{ and } n \in \zz\}$ is a divisor-closed submonoid of $S[x^{\pm 1}]^*$ whose reduced monoid is isomorphic to that of $S^*$.
\end{proof}

It is known that the class of BFDs is strictly contained in that one consisting of all integral domains satisfying ACCP. Moreover, it turns out that there are semidomains satisfying ACCP that are neither integral domains nor BFSs. Indeed, if $M := \big\langle \frac 1p \mid p \in \mathbb{P} \big\rangle$, then the semidomain $E(M)$ satisfies ACCP but it is not a BFS \cite[Example~4.15]{BG20}. 
\smallskip

We now turn our attention to the finite factorization property. For this property, the following theorem goes parallel to Theorem~\ref{thm:BF transfers}.

\begin{theorem} \label{thm:FF transfers}
	For a semidomain $S$, the following statements are equivalent.
	\begin{enumerate}
		\item[(a)] $S$ is an FFS.
		\smallskip
		
		\item[(b)] $S[x]$ is an FFS.
		\smallskip
		
		\item[(c)] $S[x^{\pm 1}]$ is an FFS.
	\end{enumerate}
\end{theorem}

\begin{proof}
	(a) $\Rightarrow$ (b):  Suppose that $S$ is an FFS, and let $K$ be a field containing $S$. Assume, by contradiction, that $S[x]$ is not an FFS. Take a nonunit $f_0 \in S[x]^*$ such that $\big\{ g S[x]^\times \mid g \in S[x] \text{ and } g \mid_{S[x]} f_0 \big\}$ is infinite (this element exists by \cite[Corollary~2]{fHK92}). Now let $(f_n)_{n \in \nn}$ be a sequence whose terms are non-associate divisors of $f_0$ in $S[x]$. For each $n \in \nn_0$, let $s_n$ be the leading coefficient of $f_n$. As $s_n \mid_{S} s_0$ for every $n \in \nn$ and the set $\big\{t S^\times \in S^*_{\text{red}} \;\mid\; t \,|_{S^*}\, s_0 \big\}$ is finite by \cite[Corollary~2]{fHK92}, after replacing $(f_n)_{n \in \nn}$ by a subsequence we can assume that $s_n S^\times = s_1 S^\times$ for every $n \in \nn$. Then we can replace $f_n$ by $s_1 s_n^{-1} f_n$ for every $n \in \nn_{\ge 2}$ and assume that each term of $(f_n)_{n \in \nn}$ has leading coefficient $s_1$. Because $f_n \mid_{K[x]} f_0$ for every $n \in \nn$ and $K[x]$ is an FFD (in fact, a UFD), we can take $i,j \in \nn$ with $i \neq j$ and $f_i K^\times = f_j K^\times$. Since both $f_i$ and $f_j$ have leading coefficient $s_1$, we see that $f_i = f_j$, contradicting that they are not associates in $S[x]$. Hence $S[x]$ is an FFS.
	\smallskip
	
	(b) $\Rightarrow$ (c): Suppose that $S[x]$ is an FFS. By \cite[Corollary~2]{fHK92}, it suffices to show that every nonzero $f \in S[x]$ with $\text{ord} \, f = 0$ has only finitely many divisors in $S[x^{\pm 1}]$ up to associates. To do so, fix a nonzero $f \in S[x]$ with $\text{ord} \, f = 0$. Now assume that $x^{d_1} g_1$ and $x^{d_2} g_2$ are divisors of $f$ in $S[x^{\pm 1}]$ for some $d_1, d_2 \in \zz$ and $g_1, g_2 \in S[x]$ with $\text{ord} \, g_1 = \text{ord} \, g_2 = 0$. Observe that $g_1$ and $g_2$ divide $f$ in $S[x]$ and also that $x^{d_1} g_1$ and $x^{d_2} g_2$ are associates in $S[x^{\pm 1}]$ if and only if $g_1$ and $g_2$ are associates in $S[x]$. As $S[x]$ is an FFS, it follows from \cite[Corollary~2]{fHK92} that $f$ has only finitely many divisors in $S[x]$ up to associates, and so our previous observation ensures that $f$ has only finitely many divisors in $S[x^{\pm 1}]$ up to associates. Therefore $S[x^{\pm 1}]$ is an FFS.
	\smallskip
	
	(c) $\Rightarrow$ (a): Suppose that $S[x^{\pm 1}]$ is an FFS. Then $S$ is an FFS because $S^*$ and the divisor-closed submonoid $\{sx^n \mid s \in S^* \text{ and } n \in \zz\}$ of $S[x^{\pm 1}]^*$ have isomorphic reduced monoids.
\end{proof}

In the class of integral domains, the bounded factorization property does not imply the finite factorization property (see, for instance, \cite[Example~4.7]{AG22}). Hence the same statement must hold in the class of semidomains. As the following example illustrates, there are positive semidomains that are BFSs but not FFSs.

\begin{example} \label{ex:BFS not FFS}
	Fix $k \in \nn_{\ge 2}$ and observe that $S = \nn_0 \cup \rr_{\ge k}$ is a positive semiring. It follows from \cite[Theorem~5.1]{BCG21} that $S$ is a BFS with $\mathcal{A}(S) = \big( \mathbb{P}_{< k^2}  \cup [k, k^2) \big) \setminus \mathbb{P} \cdot S_{> 1}$. In addition, $S$ is a reduced semiring because $\min S^* = 1$. Showing that $S$ is not an FFS amounts to observing that the equality $(k+1)^2 = (\mu (k+1))(\mu^{-1}(k+1))$ yields a factorization of $(k+1)^2$ for each $\mu \in \rr_{> 1}$ close enough to $1$ such that $k < \mu^{-1} (k+1) < \mu(k+1) < k+2 \le k^2$.
\end{example}

In light of Example~\ref{ex:power series extensions}, the statements of Theorems~\ref{thm:BF transfers} and~\ref{thm:FF transfers} are not longer true if one replaces either $S[x]$ or $S[x^{\pm 1}]$ by $S\ldb x \rdb$.
\smallskip

The diagram of Figure~\ref{fig:BF/FF transfers} summarizes the results we have established in this section.

\begin{figure}[h]
	\begin{tikzcd}[cramped]
		S \textbf{ is an FFS } \  \arrow[r, Rightarrow]	\arrow[red, r, Leftarrow, "/"{anchor=center,sloped}, shift left=1.7ex] \arrow[d, Leftrightarrow] & \ S \textbf{ is a BFS } \arrow[d, Leftrightarrow]  \\ %& S \textbf{ satisfies ACCP } \arrow[d, Leftrightarrow] \\
		S[x] \textbf{ is an FFS } \ \arrow[d, Leftrightarrow] \arrow[r, Rightarrow]	\arrow[red, r, Leftarrow, "/"{anchor=center,sloped}, shift left=1.7ex] & \ S[x] \textbf{ is a BFS } \arrow[d, Leftrightarrow] \\%\arrow[r, Rightarrow]  \arrow[red, r, Leftarrow, "/"{anchor=center,sloped}, shift left=1.7ex] %& S[x] \textbf{ satisfies ACCP}  %\arrow[r, Rightarrow] \arrow[red, r, Leftarrow, "/"{anchor=center,sloped}, shift left=1.7ex]  &
		S[x^{\pm 1}] \textbf{ is an FFS } \ \arrow[r, Rightarrow] \arrow[red, r, Leftarrow, "/"{anchor=center,sloped}, shift left=1.7ex] & \ S[x^{\pm 1}] \textbf{ is a BFS } %\arrow[d, Leftrightarrow]
	\end{tikzcd}
	\caption{As proved in Theorems~\ref{thm:BF transfers} and~\ref{thm:FF transfers}, the vertical implications in the above diagram hold for every semidomain $S$. The same theorems, along with Example~\ref{ex:BFS not FFS}, ensure that none of the horizontal implications is reversible, which is illustrated by the red marked arrows.}
	\label{fig:BF/FF transfers}
\end{figure}
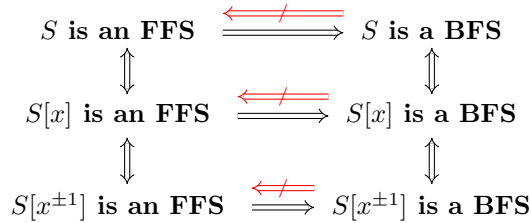

\bigskip
%%%%%%%%%%%%%%%%%
%%%%%%%%%%%%%%%%%
\section{Factoriality Properties}
\label{sec:factoriality}

It is well known that an integral domain $R$ is a UFD if and only if $R[x]$ is a UFD. However, the same result does not hold for the more general class of semidomains as indicated by the following example.

\begin{example} \label {ex: Nakayama's example}
	While the semidomain $\nn_0$ is a UFS, we will verify that the polynomial extension $\nn_0[x]$ is not. To do so, consider the polynomial $f(x) := x^5 + x^4 + x^3 + x^2 + x + 1 \in \nn_0[x]$. We can factor $f$ in $\nn_0[x]$ in the following two ways:
	\begin{equation} \label{eq:two factorizations of the same polynomials}
		f(x) = (x+1)(x^4 + x^2 + 1) \quad \text{ and } \quad f(x) = (x^2 + x + 1)(x^3 + 1).
	\end{equation}
	One can now verify that any decomposition of $x+1$, $x^2 + x + 1$, $x^3 + 1$, and $x^4 + x^2 + 1$ as a product of non-constant polynomials in $\cc[x]$ must contain a factor that does not belong to $\nn_0[x]$. Hence the decompositions in~\eqref{eq:two factorizations of the same polynomials} are actually distinct factorizations of $f$ in $\nn_0[x]$. Thus, $\nn_0[x]$ is not a UFS.
\end{example}

Recall that a semidomain $S$ is called length-factorial (or an LFS for short) if $S^*$ is an LFM; that is, $S^*$ is atomic and any two distinct factorizations of the same element of $S^*$ have distinct lengths. It was proved by Coykendall and Smith~\cite{CS11} that an integral domain is an LFS if and only if it is a UFS. As a result, the length-factorial property (somehow vacuously) ascends to (Laurent) polynomial domains. In this section, we prove that the class of semidomains where the length-factorial property ascends to (Laurent) polynomial semidomains is precisely the class consisting of integral domains. 

For the rest of the section, we identify a semidomain $S$ with a subsemiring of the integral domain $\mathcal{G}(S)$ (see Lemma~\ref{lem:characterization of integral semirings}). Given a semidomain $S$, let $n$ be the smallest positive integer such that the sum of $n$ copies of $1$ equals $0$ in $S$, and let $n$ be~$0$ if such a positive integer does not exist. As in the context of commutative rings with identities, we call~$n$ the \emph{characteristic} of $S$. 

We proceed to show that the irreducible polynomials in Example~\ref{ex: Nakayama's example} are still irreducible as polynomials in $S[x]$ for any semidomain~$S$ that is not an integral domain. Then we will take a look at two applications of this technical lemma.

\begin{lemma} \label{lemma: irreducible polynomials in Nakayama's example}
	Let $S$ be a semidomain that is not an integral domain. Then the polynomials
	\begin{equation} \label{eq:four polynomials}
		x+1, \quad x^2 + x + 1, \quad x^3 + 1, \quad \text{ and } \quad x^4 + x^2 + 1  
	\end{equation}
	are irreducible in $S[x]$.
\end{lemma}

\begin{proof}
	Note that if $S$ had finite characteristic, then every element of $S$ would have an additive inverse, which contradicts the hypothesis that $S$ is not an integral domain. Consequently, $S$ has characteristic~$0$. Also, observe that any element of $S$ dividing any of the polynomials in~\eqref{eq:four polynomials} must be a unit. Let us analyze each polynomial $p(x)$ in \eqref{eq:four polynomials} independently.
	
	\smallskip
	\noindent \textsc{Case 1:} $p(x) = x+1$. This case immediately follows from the fact that the polynomial $p(x)$ is indecomposable in $S[x]$ along with our previous observation that every constant factor of $p(x)$ in $S[x]$ is a unit.
	
	\smallskip
	\noindent \textsc{Case 2:} $p(x) = x^2 + x + 1$. Suppose, towards a contradiction, that the polynomial $p(x)$ is not irreducible in $S[x]$. As in \textsc{Case 1}, we see that $p(x)$ is not divisible in $S[x]$ by any nonunit of $S$. Then we can write $p(x) = (ax + b)(cx + d)$ for some $a,b,c,d \in S^*$, from which we obtain the identities $ad + bc = ac = bd = 1$. Thus,
	\[
		ab = ab\left((ad)(ac) + (bc)(bd)\right) = abcd(a^2 + b^2) = a^2 + b^2.
	\]
	Therefore $b^3 = ab^2 - a^2b$ in $\mathcal{G}(S)$, and we can use this identity to obtain
	\[
		b^3c^3 = (ab^2 - a^2b)c^3 = b^2c^2 - bc = bc(bc - 1) = bc(-ad) = -1.
	\]
	However, $-1 = b^3 c^3 \in S$ implies that $S$ is an integral domain, a contradiction.
	
	\smallskip
	\noindent \textsc{Case 3:} $p(x) = x^3 + 1$. Suppose, by way of contradiction, that $p(x)$ reduces in $S[x]$. Since $p(x)$ is not divisible in $S[x]$ by any nonunit of~$S$, we can write $p(x) = (ax^2 + bx + c)(dx + e)$ for some $a,b,c,d,e \in S$. Expanding the product, we obtain the identities $ce = ad = 1$ and $ae + bd = be + cd = 0$, whence
	\[
		0 = cd(ae + bd) = 1 + bcd^2.
	\]
	This implies that $-1 = bcd^2 \in S$, which contradicts that $S$ is not an integral domain.
	
	\smallskip
	\noindent \textsc{Case 4:} $p(x) = x^4 + x^2 + 1$. Suppose, by way of contradiction, that $p(x)$ reduces in $S[x]$. As $p(x)$ is not divisible in $S[x]$ by any nonunit of~$S$, it follows that $p(x)$ factors in $S[x]$ either as a polynomial of degree~$1$ times a polynomial of degree~$3$, or into two polynomials of degree~$2$, yielding the following two subcases.
	
	\smallskip
	\noindent \textsc{Case 4.1:} $p(x) = (ax^3 + bx^2 + cx + d)(ex + f)$ for some $a,b,c,d,e,f \in S$. After expanding this product, we obtain the identities $ae = df = 1$ and $cf + ed = 0$. Therefore
	\[
		0 = af(cf + ed) = acf^2 + (ae)(df) = acf^2 + 1.
	\]
	This implies that $-1 = acf^2 \in S$, which contradicts that $S$ is not an integral domain.
	
	\smallskip
	\noindent \textsc{Case 4.2:} $p(x) = (ax^2 + bx + c)(dx^2 + ex + f)$ for some $a,b,c,d,e,f \in S$. Observe that if $b = e = 0$, we can generate a contradiction by reducing this case to Case 2. Thus, we can assume, without loss of generality, that $e \neq 0$. After unfolding the product $(ax^2 + bx + c)(dx^2 + ex + f)$, we obtain the identities $ad = cf = af + be + cd = 1$ and $ae + bd = bf + ce = 0$. Since $d \neq 0$ and
	\[
		d(a^2 e + b) = (ad)(ae) + bd = ae + bd = 0,
	\]
	the equality $a^2 e + b = 0$ holds. Similarly, we can obtain the equality $c^2 e + b = 0$ (letting $f$ and $c^2e +b$ playing the roles of $d$ and $a^2e + b$, respectively). As a consequence, the fact that $e \neq 0$ ensures that $a^2 = c^2$. Hence either $a = c$ or $a = -c$ in $\mathcal{G}(S)$. Let us assume first that $a = c$. Using this assumption, we obtain that
	\[
		1 + be = (ad + cf - 1) + (1 - af - cd) = (a - c)(d - f) = 0.
	\]
	This implies that $-1 = be \in S$, which contradicts that $S$ is not an integral domain. Assume now that $a = -c$. Using this assumption, we obtain that
	\[
		3 = ad + cf + (af + be + cd) = a(d + f) + c(d + f) + be = (a + c)(d + f) + be = be.
	\]
	This implies that $-1 = 2 - be = 2 - (1 - af - cd) = af + cd + 1 \in S$, which once again contradicts the fact that $S$ is not an integral domain.
\end{proof}

\begin{prop} \label{prop:LFS polynomial semiring has coefficients in an integral domain}
	Let $S$ be a semidomain. If $S[x]$ is an LFS, then $S$ is an integral domain.
\end{prop}

\begin{proof}
	This is an immediate consequence of Lemma~\ref{lemma: irreducible polynomials in Nakayama's example}.
\end{proof}

Recall that a semidomain $S$ is called a GCD-semidomain if $S^*$ is a GCD-monoid. It is well known and not hard to verify that every UFM is a GCD-monoid. Hence every UFS is a GCD-semidomain. As another application of Lemma~\ref{lemma: irreducible polynomials in Nakayama's example}, we will characterize, for a semidomain $S$, when $S[x]$ is a GCD-semidomain.

\begin{prop} \label{prop:gcd}
	For a semidomain $S$, the following statements are equivalent.
	\begin{enumerate}
		\item[(a)] $S$ is a GCD-domain.
		\smallskip
		
		\item[(b)] $S[x]$ is a GCD-semidomain.
	\end{enumerate}
\end{prop}

\begin{proof}
	(a) $\Rightarrow$ (b): Assume that $S$ is a GCD-domain. It follows from \cite[Theorem~6.4]{GP74} that $S[x]$ is also a GCD-domain, and so $S[x]$ is a GCD-semidomain.
	\smallskip
	
	(b) $\Rightarrow$ (a): Assume now that $S[x]$ is a GCD-semidomain. Because $S^*$ is a divisor-closed submonoid of $S[x]^*$, it follows that $S^*$ is a GCD-monoid. 
	
	We are done once we argue that $S$ is an integral domain. Suppose, by way of contradiction, that this is not the case. Then it follows from Lemma~\ref{lemma: irreducible polynomials in Nakayama's example} that the polynomial $x+1$ is an irreducible element of $S[x]^*$. On the other hand, the equality $(x+1)(x^4 + x^2 + 1) = (x^2 + x + 1)(x^3 + 1)$ in~\eqref{eq: different factors same monomial} ensures that $x+1$ is not a prime in $S[x]^*$. Therefore \cite[Theorem~6.7(2)]{rG84} guarantees that $S[x]^*$ is not a GCD-monoid. However, this contradicts the fact that $S[x]$ is a GCD-semidomain. Hence $S$ is an integral domain.
\end{proof}

Unlike the property of being an MCD-semidomain, the property of being a GCD-semidomain does not ascend to polynomial semidomains.

\begin{example}
	The monoid $\nn_0$ is a UFS and, therefore, a GCD-semidomain. However, $\nn_0[x]$ is not a GCD-semidomain. Indeed, we have seen in the proof of Proposition~\ref{prop:gcd} that the polynomial $x+1$ is an irreducible element in $\nn_0[x]$ that is not prime, and so the multiplicative monoid $\nn_0[x]^*$ is not a GCD-monoid. Alternatively, it is not hard to verify that the polynomials $f(x) := (x+1)(x^4 + x^2 + 1)$ and $g(x) = (x+1)(x^3 + 1)$ do not have a greatest common divisor in $\nn_0[x]$: in fact, their only common divisors are $x+1$ and $x^3+1$, but $x+1 \nmid_{\nn_0[x]} x^3 + 1$ and $x^3 + 1 \nmid_{\nn_0[x]} x+1$.
\end{example}

Now we are in a position to characterize in several ways when $S[x]$ (or $S[x^{\pm 1}]$) is length-factorial.

\begin{theorem} \label{UFD}
	For a semidomain $S$, the following statements are equivalent.
	\begin{enumerate}
		\item[(a)] $S$ is a UFD.
		\smallskip
		
		\item[(b)] $S[x]$ is a UFS.
		\smallskip
		
		\item[(c)] $S[x^{\pm 1}]$ is a UFS.
		\smallskip
		
		\item[(d)] $S$ is an LFD.
		\smallskip
		
		\item[(e)] $S[x]$ is an LFS.
		\smallskip
		
		\item[(f)] $S[x^{\pm 1}]$ is an LFS.
	\end{enumerate}
\end{theorem}

\begin{proof}
	(a) $\Leftrightarrow$ (b): The direct implication follows from the well-known fact that the unique factorization property ascends to polynomial rings. For the reverse implication, suppose that $S[x]$ is a UFS. Then $S[x]$ is an atomic GCD-semidomain. Thus, $S$ is atomic by Theorem~\ref{thm:atomicity transfers}. In addition, it follows from Proposition~\ref{prop:gcd} that~$S$ is a GCD-domain, which concludes our argument given that every atomic GCD-domain is a UFD (see \cite[page~114]{fHK}).
	\smallskip
	
	(a) $\Leftrightarrow$ (d): It follows from \cite[Corollary~2.11]{CS11}.
	\smallskip
	
	(b) $\Leftrightarrow$ (c): Consider the multiplicative monoid $M = \{f \in S[x]^* \mid \text{ord} \, f = 0\}$, and observe that $S[x]^*$ is isomorphic to the product monoid $\nn_0 \times M$ via the map $f \mapsto (\text{ord} \, f, f/x^{\text{ord} \, f})$. It is clear that $S[x]$ is a UFS if and only if $M$ is a UFM. Now the equivalence follows from the fact that $M$ and $S[x^{\pm 1}]^*$ have isomorphic reduced monoids.
	\smallskip
	
	(b) $\Leftrightarrow$ (e): The direct implication follows from definitions. For the reverse implication, suppose that $S[x]$ is an LFS. In this case, $S$ must be an integral domain by virtue of Proposition~\ref{prop:LFS polynomial semiring has coefficients in an integral domain}. Hence $S[x]$ is an integral domain. Therefore it follows from \cite[Corollary~2.11]{CS11} that $S[x]$ is a UFD.
	\smallskip
	
	(e) $\Leftrightarrow$ (f): This follows similarly to (b) $\Leftrightarrow$ (c), after observing that, under the same notation used to prove the later, $S[x]$ is an LFS if and only if $M$ is an LFM.
\end{proof}

As we have mentioned before, the multiplicative monoid of an integral domain is a UFM if and only if it is an LFM (see~\cite{CS11}). However, a similar statement is not settled if we replace the class of integral domains for the larger class of semidomains. Motivated by this, we pose the following question.

\begin{question}\footnote{A negative answer to this question has been recently given by Bu, Vulakh, and Zhao in \cite{BVZ23}.}
	Is there an LFS that is not a UFS?
\end{question}

Recall that a semidomain $S$ is a half-factorial semidomain (or an HFS for short) provided that $S^*$ is an HFM; that is, $S^*$ is atomic and any two factorizations of the same element of $S^*$ have the same length. The half-factorial property does not behave well under polynomial extensions even in the context of integral domains (see \cite[Theorem~2.4]{aZ80} and \cite[Example~5.4]{AAZ91}). On the other hand, a semidomain $S[x]$ satisfying that $(S,+)$ is reduced is ``far" from being half-factorial in a sense that we now explain. One of the most studied arithmetic statistics related to the sets of lengths of an atomic monoid is the elasticity. The elasticity, first studied by Steffan~\cite{jlS86} and Valenza~\cite{rV90} in the context of algebraic number theory, measures the deviation of an atomic monoid from being half-factorial. Let $M$ be an %\textcolor{red}{reduced} 
atomic monoid. The \emph{elasticity} of an element $b \in M \setminus \mathcal{U}(M)$, denoted by $\rho(b)$, is defined as
\[
	\rho(b) = \frac{\sup \mathsf{L}(b)}{\inf \mathsf{L}(b)}.
\]
By convention, we set $\rho(u) = 1$ for every $u \in \uu(M)$. Observe that $\rho(b) \in \qq_{\ge 1} \cup \{\infty\}$ for all $b \in M$. In addition, the \emph{elasticity} of the whole monoid $M$ is defined to be
\[
	\rho(M) := \sup \{\rho(b) \mid b \in M \}.
\]
Observe that a monoid $M$ is an HFM if and only if $\rho(M) = 1$. Thus, we can think of the elasticity as an arithmetic statistic to measure how far an atomic monoid is from being an HFM, in which case, an atomic monoid having infinite elasticity is as far from being half-factorial as it can possibly be. On the other hand, the \emph{set of elasticities} of $M$ is $R(M) := \{\rho(b) \mid b \in M\}$, and $M$ is said to have \emph{full elasticity} provided that $R(M) = \qq \cap [1, \rho(M)]$. As a monoid is half-factorial if and only if its set of elasticities is a singleton, namely $\{1\}$, we observe that for a given monoid having full elasticity provides an indication that it is as far from being an HFM as it can possibly be.

The following proposition is a generalized version of \cite[Theorem~2.3]{CCMS09}, and here we adapt the proof given in \cite{CCMS09} to fit the more general setting of polynomials over semidomains.

\begin{prop} \label{prop: elasticity of reduced integral semirings}
	Let $S$ be a semidomain such that $S[x]$ is atomic. Then $S[x]$ has full and infinite elasticity provided that $(S,+)$ is reduced. 
\end{prop}

\begin{proof}
	As we pointed out before, if $S$ has finite characteristic, then $(S,+)$ is not reduced (it is, in fact, a group). Consequently, $S$ must have characteristic~$0$. Now let $K$ be a field containing $S$ as a subsemiring. We first claim that for every $n \in \nn_{\ge 2}$, the polynomial $(x + n)^n (x^2 - x + 1)$ is an irreducible element in $S[x]$. It follows from \cite[Lemma~2.1]{CCMS09} that for every $m \in \nn_0$ the polynomial $(x + n)^m (x^2 - x +1)$ belongs to $\nn_0[x]$ if and only if $m \ge n$. This, together with the fact that $S$ is a semidomain whose additive monoid is reduced, guarantees that $(x + n)^m (x^2 - x + 1) \notin S[x]$ when $m < n$. Therefore the fact that $K[x]$ is a UFD guarantees that $(x+n)^n(x^2 - x + 1)$ is an irreducible element in $S[x]$.
	
	By Lemma~\ref{lemma: irreducible polynomials in Nakayama's example}, the polynomial $(x^2 - x + 1)(x+1) = x^3 + 1$ is irreducible in $S[x]$. Now for $n, k \in \nn$, consider the polynomial
	\[
		f(x) := (x+n)^n(x^2 - x + 1)(x + 1)^k \in \nn_0[x] \subseteq S[x].
	\]
	As every divisor of $f(x)$ in $S[x]$ is a divisor of $f(x)$ in $K[x]$ and $K[x]$ is a UFD, the only two factorizations of $f(x)$ in $S[x]$ are $[(x+n)^n (x^2 - x + 1)] \cdot [x+1]^k$ and $[x+n]^n \cdot [(x^2 - x + 1)(x+1)] \cdot [x+1]^{k-1}$, which have lengths $k+1$ and $k+n$, respectively. Since $\{ (k+n)/(k+1) \mid k,n \in \nn \} = \qq_{\ge 1}$, we conclude that $S[x]$ has full and infinite elasticity.
\end{proof}

\bigskip
%%%%%%%%%%%%%%%%
%%%%%%%%%%%%%%%%
\section*{Acknowledgments}

During the preparation of this paper, the first author was supported by the NSF awards DMS-1903069 and DMS-2213323, while the second author was supported by the University of Florida Mathematics Department Fellowship. The authors kindly thank an anonymous referee for helpful suggestions.

\bigskip
%%%%%%%%%%%%%%%
%%%%%%%%%%%%%%%

\end{document}